\documentclass[12pt]{article}
\usepackage[centertags]{amsmath}
\usepackage{amsfonts}
\usepackage{amssymb}
\usepackage{latexsym}
\usepackage{amsthm}
\usepackage{newlfont}
\usepackage{graphicx}
\usepackage{listings}
\usepackage{booktabs}
\usepackage{abstract}
\usepackage{enumerate}
\usepackage{xcolor}
\RequirePackage{srcltx}
\date{}

\newlength{\defbaselineskip}
\newcommand{\setlinespacing}[1]%
           {\setlength{\baselineskip}{#1 \defbaselineskip}}

\newcommand{\N}{{\mathbb{N}}}

\newcommand{\actaqed}{\hfill $\actabox$}
{\medskip\noindent \textit{Proof of #1. }}%
{\actaqed \medskip}

\def\C{{\mathcal C}}
\def\cC{{\mathcal C}}

\def \Tr{\mathcal T}

\def \cN{\mathcal N}

\def \cM{\mathcal M}
\def\R{{\mathbb R}}

\def \T{\mathbb T}

\def\bbC{\mathbb C}
\def \<{\langle}
\def\>{\rangle}

\def\Og{\Omega}
\def \ep{\epsilon}
\def \e{\varepsilon}

\def\al{\alpha}

\def \ro{\varrho}

\def \sp{\operatorname{span}}

\def\bx{\mathbf x}
\def\by{\mathbf y}

\def\bw{\mathbf w}

\def\bN{\mathbf N}
\def\bX{\mathbf X}

\def\bF{\mathbf F}

\def\Ld{\Lambda}
\def\Og{\Omega}
\def\og{\omega}

\newtheorem{Theorem}{Theorem}[section]
\newtheorem{Lemma}{Lemma}[section]

\newtheorem{Proposition}{Proposition}[section]
\newtheorem{Remark}{Remark}[section]

\newtheorem{Corollary}{Corollary}[section]
\numberwithin{equation}{section}

\newcommand{\be}{\begin{equation}}
\newcommand{\ee}{\end{equation}}

\begin{document}

\title{Sampling discretization of integral norms and its application}

\author{ F. Dai and   V. Temlyakov 	\footnote{
		The first named author's research was partially supported by NSERC of Canada Discovery Grant
		RGPIN-2020-03909.
		The second named author's research was supported by the Russian Federation Government Grant No. 14.W03.31.0031.
  }}

\newcommand{\Addresses}{{
  \bigskip
  \footnotesize

  F.~Dai, \textsc{ Department of Mathematical and Statistical Sciences\\
University of Alberta\\ Edmonton, Alberta T6G 2G1, Canada\\
E-mail:} \texttt{fdai@ualberta.ca }

 \medskip
  V.N. Temlyakov, \textsc{University of South Carolina,\\ Steklov Institute of Mathematics,\\  Lomonosov Moscow State University,\\ and Moscow Center for Fundamental and Applied Mathematics.
  \\
E-mail:} \texttt{temlyak@math.sc.edu}

}}
\maketitle

\begin{abstract}
	{The paper addresses a problem of sampling discretization of integral norms of elements of finite-dimensional subspaces satisfying some conditions. We prove sampling discretization results under two standard kinds of assumptions -- conditions on the entropy numbers and  conditions in terms of the Nikol'skii-type inequalities. We prove some upper bounds on the number of sample points sufficient for good discretization and show that these upper bounds are sharp in a certain sense. Then we apply our general conditional results to subspaces with special structure, namely,
subspaces with the tensor product structure. We demonstrate that applications of results based on 
the Nikol'skii-type inequalities provide somewhat better results than applications of results based on 
 the entropy numbers conditions. Finally, we apply discretization results to the problem of sampling recovery. }
\end{abstract}


{\it MSC classification 2000:} Primary 65J05; Secondary 42A05, 65D30, 41A63.

\section{Introduction} 

Let $\Omega$ be a subset of $\R^d$ with the probability measure $\mu$. By $L_p$, $1\le p< \infty$, norm we understand
$$
\|f\|_p:=\|f\|_{L_p(\Omega,\mu)} := \left(\int_\Omega |f|^pd\mu\right)^{1/p}.
$$
By $L_\infty$ norm we understand the uniform norm of continuous functions
$$
\|f\|_\infty := \max_{\bx\in\Omega} |f(\bx)|
$$
and with some abuse of notation we occasionally write $L_\infty(\Omega)$ for the space $\cC(\Omega)$ of continuous functions on $\Omega$.

By discretization of the $L_p$ norm we understand a replacement of the measure $\mu$ by
a discrete measure $\mu_m$ with support on a set $\xi =\{\xi^j\}_{j=1}^m \subset \Omega$   in such a way that the error $|\|f\|^p_{L_p(\Omega,\mu)} - \|f\|^p_{L_p(\Omega,\mu_m)}|$ is small for functions from a given class. In this paper we focus on discretization of the $L_p$ norms of elements of finite-dimensional subspaces. Namely, we work on the following problem. 

{\bf The Marcinkiewicz discretization problem.} Let $\Omega$ be a compact subset of $\R^d$ with the probability measure $\mu$. We say that a linear subspace $X_N$ (index $N$ here, usually, stands for the dimension of $X_N$) of $L_p(\Omega,\mu)$, $1\le p < \infty$, admits the Marcinkiewicz-type discretization theorem with parameters $m\in \N$ and $p$ and positive constants $C_1\le C_2$ if there exists a set  
$$
\Big\{\xi^j \in \Omega: j=1,\dots,m\Big\}
$$
 such that for any $f\in X_N$ we have
\be\label{I1}
C_1\|f\|_p^p \le \frac{1}{m} \sum_{j=1}^m |f(\xi^j)|^p \le C_2\|f\|_p^p.
\ee

{\bf The Bernstein discretization problem.} In the case $p=\infty$ we define $L_\infty$ as the space of continuous functions on $\Omega$  and ask for
\begin{equation}\label{I2}
C_1\|f\|_\infty \le \max_{1\le j\le m} |f(\xi^j)| \le  \|f\|_\infty.
\end{equation}

We will also use the following brief way to express the above properties: The $\cM(m,p)$ (more precisely the $\cM(m,p,C_1,C_2)$) theorem holds for  a subspace $X_N$, written $X_N \in \cM(m,p)$ (more precisely $X_N \in \cM(m,p,C_1,C_2)$). In the case $q=\infty$ we always have $C_2=1$ and for brevity we write $\cM(m,\infty,C_1)$ instead of $\cM(m,\infty,C_1,1)$.

There are known results on the Marcinkiewicz discretization problem proved for subspaces $X_N$
satisfying some conditions. There are two types of conditions used in the literature: (I) Conditions on the entropy numbers and (II) Conditions in terms of the Nikol'skii-type inequalities. We now describe these conditions in detail. 

{\bf I. Entropy conditions.} We begin with the definition of the entropy numbers.
  Let $X$ be a Banach space and let $B_X$ denote the unit ball of $X$ with the center at $0$. Denote by $B_X(y,r)$ the ball with center $y$ and radius $r>0$,  $\{x\in X:\|x-y\|\le r\}$. For a compact set $A\subset X$  we define the entropy numbers $\e_k(A,X)$, $k=0,1,\cdots$: 
$$
\e_k(A,X)  :=\inf \Bigl\{\e>0 : \exists y^1,\dots ,y^{2^k} \in A ,\    A \subseteq \cup_{j=1}
^{2^k} B_X(y^j,\e)\Bigr\}.
$$
In our definition of   $\e_k(A,X)$ we require $y^j\in A$. In a standard definition of $\e_k(A,X)$ this restriction is not imposed. 
However, it is well known (see \cite{VTbook}, p.208) that these characteristics may differ at most by the factor $2$. Throughout the paper we use the following notation for the unit $L_q$ ball of $X_N$
$$
X_N^q:= \{f\in X_N:\, \|f\|_q \le 1\}.
$$
Here is a standard  entropy assumption in discretization theory: Suppose that a subspace $X_N$ satisfies the condition $(B\ge 1)$
\be\label{I3}
\e_k(X^q_N,L_\infty) \le  B  (N/k)^{1/q},\quad 1\le k\le N.
 \ee

{\bf II. Nikol'skii-type inequalities.} Let $q\in [1,\infty)$ and $X_N\subset L_\infty(\Omega)$. The inequality
\begin{equation}\label{I4}
\|f\|_\infty \leq M\|f\|_q,\   \ \forall f\in X_N
\end{equation}
is called the Nikol'skii inequality for the pair $(q,\infty)$ with the constant $M$. 
It is convenient to write the constant $M$ in the form $M=BN^{1/q}$. If $X_N$ satisfies 
(\ref{I4}) with $M=BN^{1/q}$, then we say that $X_N$ satisfies Condition NqB (see Section \ref{sr}
below) and write $X_N \in \cN(q,\infty,B)$.

 It is well known that in the case $q\in [2,\infty)$ the above entropy condition and Nikol'skii-type inequality are closely related. We now comment on the relation between the Nikol'skii inequality and
the entropy condition (\ref{I3}). On one hand it is easy to see (see \cite{DPSTT2}) that the entropy condition (\ref{I3}) for $k=1$ implies the following Nikol'skii inequality
$$
\|f\|_\infty \le 4BN^{1/q}\|f\|_q \quad \text{for all}\quad f\in X_N.
$$
On the other hand we note that the Nikol'skii-type
inequality condition \eqref{I4} with $M=B N^{1/q}$ implies
 the entropy condition \eqref{I3} with $k=1$. Thus, the Nikol'skii-type
inequality condition is equivalent to the entropy condition (\ref{I3}) for $k=1$.
Moreover, Lemma \ref{AL2} (see below)   from
\cite{VT180} shows that in the case $q\in [2,\infty)$ the Nikol'skii inequality (\ref{I4}) combined with an extra mild condition implies  the entropy condition (\ref{I3}) with $B$ replaced by $C(\log N)^{1/q}B$.

We refer the reader to the recent survey papers \cite{DPTT} and \cite{KKLT} for a detailed 
description of known results on the Marcinkiewicz discretization problem under conditions I and II. 
In this paper we only cite those results which are directly related to our new results. We now give 
brief comments on results obtained in Sections \ref{A}--\ref{sr}. 

In Section \ref{A} we discuss the following setting. Assume that parameters $1\le p,q <\infty$ are given. We would like to solve the Marcinkiewicz discretization problem for the $L_p$ norm under assumption that either Conditions I or II is satisfied with the parameter $q$. Typically, the known results 
address either the case $p=q$ or the case when the Nikol'skii-type inequality for the pair $(2,\infty)$ is imposed. In Section \ref{A} (see Corollary \ref{AC1}) we show how a simple argument allows us to derive a discretization result in the case $1\le q\le p<\infty$ under assumptions I from the corresponding result for $p=q$. The main result of Section \ref{A} is Theorem \ref{AT2}, which 
improves known results in the case when  $p=q$ and  $3<q<\infty$ under Conditions II. Then we use 
Theorem \ref{AT2} to  deduce  discretization results in the case $2\le q\le p <\infty$ (see Corollary \ref{AC2}). 

In Section \ref{B} we apply results of Section \ref{A} to subspaces with special structure, namely,
subspaces with the tensor product structure. We demonstrate that applications of results based on 
Conditions II provide somewhat better results than applications of results based on 
Conditions I. This observation is based on a known result on the Nikol'skii-type inequalities for subspaces with the tensor product structure (see Lemma \ref{BL1} below).

In Section \ref{sr} we apply discretization results of Section \ref{A} to the problem of sampling recovery. Recently, it was observed in \cite{VT183} how discretization results can help to prove 
general inequalities between optimal sampling recovery and the Kolmogorov widths. Namely,
it was proved in \cite{VT183} that the optimal error of recovery in the $L_2$ norm of functions from a class $\bF$ can be bounded above by the value of the Kolmogorov width of $\bF$ in the uniform norm. In Section \ref{sr} we demonstrate how to derive some general inequalities for the optimal sampling recovery in $L_p$ from the corresponding discretization results of Section \ref{A}.

\section{A generalization} 
\label{A}

We start with the following 
 Theorem \ref{AT1}, which was proved in \cite{VT159} in the case of  $q=1$  with the help of the chaining technique, and in \cite{DPSTT1} for the general case. 

\begin{Theorem}\label{AT1} Let $1\le q<\infty$. Suppose that a subspace $X_N$ satisfies the condition
\be\label{A1}
\e_k(X^q_N,L_\infty) \le  B (N/k)^{1/q}, \quad 1\leq k\le N,
\ee
where $B\ge 1$.
Then for a large enough constant $C(q)$ there exists a set of
$$
m \le C(q)B^{q}N(\log_2(2BN))^2
$$
 points $\xi^j\in \Omega$, $j=1,\dots,m$,   such that for any $f\in X_N$
we have
$$
\frac{1}{2}\|f\|_q^q \le \frac{1}{m}\sum_{j=1}^m |f(\xi^j)|^q \le \frac{3}{2}\|f\|_q^q.
$$
\end{Theorem}

Here  is a direct corollary of Theorem \ref{AT1}.

\begin{Corollary}\label{AC1} Let $1\le q \le p<\infty$. Suppose that the  condition (\ref{A1}) is satisfied. Then for a large enough constant $C(p,q)$ there exists a set of
$$
m \le C(p,q)B^{p}N^{p/q} (\log_2(2BN))^2
$$
 points $\xi^j\in \Omega$, $j=1,\dots,m$,   such that for any $f\in X_N$
we have
$$
\frac{1}{2}\|f\|_p^p \le \frac{1}{m}\sum_{j=1}^m |f(\xi^j)|^p \le \frac{3}{2}\|f\|_p^p.
$$
\end{Corollary}
\begin{proof} 
  Let $1\le q\le p <\infty$. Then condition (\ref{A1})
implies
\be\label{A2}
\e_k(X^p_N,L_\infty)\le  \e_k(X^q_N,L_\infty) \le  BN^{1/q-1/p} (N/k)^{1/p}, \quad 1\leq k\le N.
\ee
Thus, condition (\ref{A1}) is satisfied for $p$ with $B'=BN^{1/q-1/p}$. Theorem \ref{AT1} gives
$$
m \le C(p)(B')^{p}N(\log_2(2B'N))^2 \le C(p,q)B^{p}N^{p/q}(\log_2(2BN))^2.
$$
\end{proof}

We now prove an analog of Theorem \ref{AT1} for $q\in (2,\infty)$ under a condition on $X_N$ in terms of the Nikol'skii inequality instead of the entropy condition (\ref{A1}) in Theorem \ref{AT1}. 

\begin{Theorem}\label{AT2} Let $2\le q<\infty$. Suppose that a subspace $X_N$ satisfies the Nikol'skii type inequality
\be\label{A3}
\|f\|_\infty \le  B N^{1/q}\|f\|_q, \quad \forall f\in X_N,
\ee
where $B\ge 1 $.
Then for any $\ep\in(0,1)$ there is a large enough constant $C=C(p, q,\ep)$ such that there exists a set of
\begin{equation}\label{A2.4}
m \le C B^{q}N(\log_2 (2BN))^3 
\end{equation}
points $\xi^j\in \Omega$, $j=1,\dots,m$,   such that for any $f\in X_N$
we have
$$
(1-\ep)\|f\|_q^q \le \frac{1}{m}\sum_{j=1}^m |f(\xi^j)|^q \le (1+\ep)\|f\|_q^q.
$$
\end{Theorem}
\begin{proof} We derive Theorem \ref{AT2} from the $\ep$-version of Theorem \ref{AT1}. 
The following Remark \ref{AR1} is from \cite{DPSTT1}.
\begin{Remark}\label{AR1}
Under the assumptions of Theorem \ref{AT1},  we can deduce a slightly stronger result, namely,   that  for any $\epsilon\in (0,1)$,  there exists a set of $m$ points $\{\xi^j\}_{j=1}^m\subset \Omega$ with
$$
m \le C(q,\epsilon)B^{q}N(\log_2(2BN))^2
$$
	 such that
	\begin{equation*}
	(1-\epsilon)\|f\|_q^q\le \frac{1}{m}\sum_{j=1}^m |f(\xi^j)|^q \le (1+\epsilon)\|f\|_q^q,\   \qquad \forall f\in X_N,
	\end{equation*}
where  $C(q, \epsilon)$ is a positive constant  depending  only on $\epsilon$ and $q$.	
 \end{Remark}

At the first step we use the following Lemma \ref{AL1} from \cite{DPSTT2}. 

\begin{Lemma}\label{AL1} Let $1\leq p<\infty$  be a fixed number. Assume that   $X_N$ is  an $N$-dimensional subspace of $L_\infty(\Og)$ satisfying  the following condition  for  some  parameter $\beta >0$ and constant $K\ge 2$:
	\begin{equation}\label{A4}
	\|f\|_\infty \leq (KN)^{\frac{\beta}{p}} \|f\|_p,\   \   \forall f\in X_N.
	\end{equation}
	Let $\{\bx^j\}_{j=1}^\infty$ be a sequence of independent random points  selected   from $\Omega$ according to $\mu$.
	Then   there exists a  positive   constant  $C_\beta$  depending only on $\beta$  such that for any   $0< \ep\leq \frac 12$ and    \begin{equation}\label{A5}
	m\ge C_\beta K^\beta \ep^{-2}( \log( 2/\ep)) N^{\beta+1}\log N, 
	\end{equation}  
	the   inequality
	\begin{align}
(1-\ep) \|f\|_p^p\leq  \frac{ 1}{m} \sum_{j=1}^m| f(\bx^j)|^p \leq (1+\ep)\|f\|_p^p,
	\end{align}
	holds for all $f\in X_N$  with probability
	$ \ge 1-m^{-N/\log K}$.
	
\end{Lemma}

By Lemma \ref{AL1} with $p=q$, $\beta =1$, $K=B^q$, and $m=S$ we replace the $L_q(\Omega,\mu)$ space 
by the $L_q(\Omega_S,\mu_S)$ with $\Omega_S= \{\bx^j\}_{j=1}^S$, $\mu_S(\bx^j)=1/S$, $j=1,\dots, S$ with the relations
\be\label{A6}
(1-\ep)\|f\|_{L_q(\Omega,\mu)}^q \le \|f\|_{L_q(\Omega_S,\mu_S)}^q \le (1+\ep)\|f\|_{L_q(\Omega,\mu)}^q,
\ee
\be\label{A7}
 S\le C_2B^q \ep^{-2}( \log( 2/\ep)) N^{2}\log N.
\ee

At the second step we apply the following Lemma \ref{AL2} from \cite{VT180}, which is based on the corresponding results from \cite{Kos}, with $s=S$ to the restriction $X_N(\Omega_S)$ of the 
subspace $X_N$ onto the $\Omega_S$ in the case of $L_q(\Omega_S,\mu_S)$. 

\begin{Lemma}\label{AL2} Let $q\in [2,\infty)$. Assume that for any $f\in X_N$ we have
\be\label{A8}
\|f\|_\infty \le BN^{1/q}\|f\|_q
\ee
with some constant $B\ge 1$. Also, assume that $X_N \in \cM(s,\infty,C_1)$ with $s\ge 1$.
Then for $k\in [1,N]$ we have  
\be\label{A9}
\e_k(X_N^q, L_\infty) \le C(q,C_1)(\log s)^{1/q} B(N/k)^{1/q} .
\ee
\end{Lemma}

This gives us the bound
\be\label{A10}
\e_k(X_N^q(\Omega_S), L_\infty) \le C(q,\ep,C_2)B (\log (2BN))^{1/q} (N/k)^{1/q} .
\ee

At the third step we apply Remark \ref{AR1} to the $X_N(\Omega_S)$ and complete the proof.

\end{proof}

Here  is a direct corollary of Theorem \ref{AT2}.

\begin{Corollary}\label{AC2} Let $2\le q \le p<\infty$. Suppose that condition (\ref{A3}) is satisfied
with $B\ge 1$. Then for any $\ep\in(0,1)$ there is a large enough constant $C=C(p,q,\ep)$ such that there exists a set of
\be\label{m}
m \le C B^{p}N^{p/q}(\log_2(2BN))^3 
\ee
 points $\xi^j\in \Omega$, $j=1,\dots,m$,   such that for any $f\in X_N$
we have
$$
(1-\ep)\|f\|_p^p \le \frac{1}{m}\sum_{j=1}^m |f(\xi^j)|^p \le (1+\ep)\|f\|_p^p.
$$
\end{Corollary}
\begin{proof} 
  Let $2\le q\le p <\infty$. Then condition (\ref{A3})
implies
\be\label{A11}
\|f\|_\infty\le B N^{1/q}\|f\|_q \le  BN^{1/q-1/p} N^{1/p}\|f\|_p, \quad 1\leq k\le N.
\ee
Thus, condition (\ref{A3}) is satisfied for $p$ with $B'=BN^{1/q-1/p}$. Theorem \ref{AT2} gives
$$
m \le C(p,\ep)(B')^{p}N(\log_2((2B'N)))^3 \le C(p, q,\ep)B^{p}N^{p/q}(\log_2(2BN) )^3.
$$
\end{proof}

\begin{Remark}\label{AR2} Corollary \ref{AC2} provides bound (\ref{m}) on $m$ that consists of
three factors: $N^{p/q}$, $(\log_2 (2BN))^3$, which grow with $N$, and $B^p$, which may grow with $N$.
We do not know if the factor $(\log_2 (2BN))^3$ can be dropped in (\ref{m}). However, we know that in the case 
$q>2$ neither $N^{p/q}$ can be replaced by $N^{p/q-\delta}$ nor $B^p$ can be replaced by $B^{p-\delta}$ with any $\delta>0$. In the case $q=2$ the factor $N^{p/2}$ cannot be replaced by $N^{p/2-\delta}$ with any $\delta>0$.
\end{Remark}
\begin{proof} We begin with the case $q=2$.  Let $\Lambda_n = \{k_j\}_{j=1}^n$ be a lacunary sequence: $k_1=1$, $k_{j+1} \ge bk_j$, $b>1$, $j=1,\dots,n-1$. Denote
$$
\Tr(\Lambda_n):= \left\{f\,:\, f(x) = \sum_{k\in \Lambda_n} c_ke^{ikx},\quad x\in\T \right\}.  
$$
It is clear that $\Tr(\Lambda_n) \in \cN(2,\infty,1)$. Indeed, for $f(x) = \sum_{k\in \Lambda_n} c_ke^{ikx}$ we have
\be\label{Anik1}
\|f\|_\infty \le \sum_{k\in \Lambda_n} |c_k| \le n^{1/2}\left(\sum_{k\in \Lambda_n} |c_k|^2\right)^{1/2} = n^{1/2}\|f\|_2.
\ee
It is proved in \cite{KKLT} (see {\bf D.20. A lower bound}) that the condition $\Tr(\Lambda_n) \in 
\cM(m,p,C_1,C_2)$ with $p>2$ and fixed $C_1$ and $C_2$ implies $ m\ge C(p,C_1,C_2) n^{p/2}$. 
This completes the proof of Remark \ref{AR2} in the case $q=2$. 

We now let $q>2$ and rewrite (\ref{Anik1}) in the form
$$
\|f\|_\infty \le n^{1/2}\|f\|_2 \le n^{1/2-1/q} n^{1/q}\|f\|_q.
$$
Therefore, $\Tr(\Lambda_n) \in \cN(q,\infty,n^{1/2-1/q})$. Then, on one hand we know from the above that 
$m$ must grow in the sense of order as $n^{p/2}$, on the other hand if we replace in (\ref{m}) $N^{p/q}$ by $N^{p/q-\delta}$, then we obtain 
$$
m \le C n^{p(1/2-1/q)}n^{p/q-\delta} (\log n)^3 \le Cn^{p/2-\delta}(\log n)^3.
$$
We get a contradiction. In the same way we get a contradiction in the case of replacement $B^p$  by $B^{p-\delta}$. The proof is complete.

\end{proof}

{\bf A comment.} Sometimes it is convenient to have the entropy bound  (\ref{A1}) for all $k\in \N$, 
namely, the bound
\be\label{A1inf}
\e_k(X^q_N,L_\infty) \le  B_1 (N/k)^{1/q}, \quad 1\leq k <\infty. 
\ee
We now prove that (\ref{A1}) implies (\ref{A1inf}) with $B_1=6B$.
We begin by pointing out that the assumption \eqref{A1} for $k=N$ implies the inequality
\begin{equation}\label{A12}
\e_k(X_N^q, L_\infty)\le 6B2^{-k/N}  \   \  \text{ for $k>N$}.
\end{equation}
This  follows directly  from the facts that for each Banach  space $X$ (see \cite[(7.1.6), p. 323]{VTbookMA}),
$$
\e_k(A,X) \le \e_N(A,X)\e_{k-N}(B_X,X), \  \ k>N,
$$
and for each $N$-dimensional space $X$ (see \cite[Corollary 7.2.2, p. 324]{VTbookMA}),
$$
\e_m(B_X,X) \le 3(2^{-m/N}).
$$
Next, we use the inequality $2^x\ge 1+x$, for $x\ge 1$. This inequality follows from ($x\ge 1$, $a=\ln 2$)
$$
e^{ax} = 1+ax+\frac{(ax)^2}{2!} +\dots \ge 1+x\left(a+\frac{a^2}{2!}+\dots\right) = 1+x(e^a-1) =1+x.
$$
Therefore, condition (\ref{A1}) implies condition (\ref{A1inf}) with $B_1= 6B$. 
 
\section{Discretization in subspaces with tensor product structure}
\label{B}

Let $s\in\N$. Suppose that we have $s$ subspaces $X(N_i,i)\subset \cC(\Omega_i)$
 with $\dim X(N_i,i) = N_i$, $i=1,\dots,s$. Denote $\bN := (N_1,\dots,N_s)$ and 
 $$
 \bX_{\bN} := \sp\{ f_1(\bx^1)\times\cdots\times f_s(\bx^s)\,:\, f_i\in X(N_i,i), i=1,\dots,s\}
 $$
 a subspace of $\cC(\Omega_1\times\cdots\times\Omega_s)$. Consider a product measure $\mu =\mu_1\times\cdots\times\mu_s$ on $\Omega:=\Omega_1\times\cdots\times\Omega_s$ with $\mu_i$ being a probability measure on $\Omega_i$, $i=1,\dots,s$. First, we prove some discretization results 
 for the $\bX_{\bN}$ under the conditions
 \be\label{B1}
\e_k(X^q(N_i,i),L_\infty) \le  B_i (N_i/k)^{1/q}, \quad 1\leq k <\infty, \quad i=1,\dots,s,
\ee
where $B_i\ge 1$, $i=1,\cdots, s$.
Note, that as it is explained at the end of Section \ref{A}, conditions (\ref{B1}) are equivalent to the same conditions with a weaker restrictions on $k$: instead of 
$1\leq k <\infty$ we can take $1\leq k\le N_i$. We begin with a simple observation.

\begin{Proposition}\label{BP1} Let $1\le q \le p<\infty$. Suppose that conditions (\ref{B1}) are satisfied. Then for a large enough constant $C(p,q,s)$ there exists a set $\xi(m)$ with a tensor product structure of
$$
m \le C(p,q,s)\prod_{i=1}^sB_i^{p}N_i^{p/q} (\log_2(2B_iN_i))^2
$$
 points $\xi^j\in \Omega$, $j=1,\dots,m$,   such that for any $f\in \bX_\bN$
we have
\be\label{B2}
\left(\frac{1}{2}\right)^s\|f\|_p^p \le \frac{1}{m}\sum_{j=1}^m |f(\xi^j)|^p \le \left(\frac{3}{2}\right)^s\|f\|_p^p.
\ee
\end{Proposition}
\begin{proof} By Corollary \ref{AC1} for each subspace $X(N_i,i)$, $i=1,\dots,s$, we find a set $\xi(m_i,i)=\{\xi^j(m_i,i)\}_{j=1}^{m_i}\subset \Omega_i$ of $m_i$ points such that 
\be\label{B3}
m_i \le C(p,q)B_i^{p}N_i^{p/q} (\log_2(2B_iN_i))^2
\ee
 and for any $f\in X(N_i,i)$ we have
\be\label{B4}
\frac{1}{2}\|f\|_p^p \le \frac{1}{m_i}\sum_{j=1}^{m_i} |f(\xi^j(m_i,i))|^p \le \frac{3}{2}\|f\|_p^p.
\ee
Then applying (\ref{B4}) successively with respect to $i=1,\dots,s$ we obtain for $f\in \bX_\bN$ inequalities (\ref{B2}) with $\xi(m) = \xi(m_1,1)\times\cdots\times\xi(m_s,s)$
and by (\ref{B3})
$$
m=\prod_{i=1}^s m_i \le C(p,q)^s\prod_{i=1}^sB_i^{p}N_i^{p/q} (\log_2(2B_iN_i))^2.
$$
This completes the proof. 
\end{proof}

Second, we discuss some results, when instead of conditions (\ref{B1}) we impose conditions in terms of Nikol'skii-type inequalities: For any $f\in X(N_i,i)$
 \be\label{B5}
\|f\|_\infty  \le  B_i N_i^{1/q}\|f\|_q,   \quad i=1,\dots,s.
\ee
 	
In the same way as Proposition \ref{BP1} was derived from the discretization result -- Corollary \ref{AC1} -- we derive from Corollary \ref{AC2} the following statement.

\begin{Proposition}\label{BP2} Let $2\le q \le p<\infty$. Suppose that conditions (\ref{B5}) are satisfied with $B_i\ge 1 $, $i=1,\dots,s$. Then for a  large enough constant $C=C(p,q,s)$ there exists a set $\xi(m)$ with a tensor product structure of
\be\label{B8m}
m \le C \prod_{i=1}^sB_i^{p}N_i^{p/q} (\log_2(2B_i  N_i) )^3
\ee
 points $\xi^j\in \Omega$, $j=1,\dots,m$,   such that for any $f\in \bX_\bN$
we have
\be\label{B9}
\left(\frac{1}{2}\right)^s\|f\|_p^p \le \frac{1}{m}\sum_{j=1}^m |f(\xi^j)|^p \le \left(\frac{3}{2}\right)^s\|f\|_p^p.
\ee
\end{Proposition}	 
	 
In a particular case when $N_i=N$ and $B_i=B$ for  $i=1,\dots,s$, the extra logarithmic factor in (\ref{B8m}) will be of order $(\log(2B  N))^{3s}$. We now show how it could be reduced to 
$(\log (2BN))^{3}$. 
We need the following lemma, which is a particular case of Theorem 3.3.3 from \cite[p.107]{VTbookMA}  in the  case of periodic functions with $\Omega_i = \T$, $i=1,\dots,s$.

\begin{Lemma}\label{BL1} Let $1\le q<\infty$. Suppose that conditions (\ref{B5}) are satisfied. Then for $f\in \bX_\bN$ we have
$$
\|f\|_\infty \le \left(\prod_{i=1}^s (B_iN_i^{1/q})\right)\|f\|_q.
$$
\end{Lemma}
\begin{proof}
 Let $f\in \bX_\bN$.  Then  for any $\bx=(\bx^1, \bx^2) \in\Omega_1\times \Omega_2$, we have 
\begin{align*}
|f(\bx^1, \bx^2)|^q & \leq B_1^q N_1 \int_{\Omega_1} |f(\by^1, \bx^2)|^q\, d\mu_1(\by^1)\\
&\leq  B_1^q N_1 B_2^q N_2  \int_{\Omega_1} \Bigl[ \int_{\Omega_2} |f(\by^1, \by^2)|^qd\mu_2(\by^2) \Bigr]\, d\mu_1(\by^1),
\end{align*}
 where we used  (\ref{B5}) for $i=1$ and  the fact  that $f(\cdot, \bx^2) \in X(N_1, 1)$ in the first step, and  (\ref{B5}) for $i=2$ and  the fact  that $f(\by^1, \cdot) \in X(N_2, 2)$ for each fixed  $\by^1\in\Omega_1$ in the second step.
This proves the stated inequality for $s=2$. The inequality  for  the general case  $s\ge 2$ follows by induction. 

\end{proof}
 
 Using Lemma \ref{BL1}, we may apply  Corollary \ref{AC2} to the space $X_N:=\bX_{\bN}$ with  $N=\prod_{i=1}^s N_i$ and $B=\prod_{i=1}^s B_i$. We then obtain  the following version of Proposition \ref{BP2}.
 
 \begin{Proposition}\label{BP3} Let $2\le q \le p<\infty$. Suppose that the  conditions (\ref{B5}) are satisfied with $B_i\ge 1$, $i=1,\dots,s$. Then for a large enough constant $C=C(p,q)$ there exists a set $\xi(m)$  of
\be\label{B10}
m \le C \left(\prod_{i=1}^sB_i^{p}N_i^{p/q}\right) \left(\log_2(\prod_{i=1}^s(B_i  N_i))\right)^3
\ee
 points $\xi^j\in \Omega$, $j=1,\dots,m$,   such that 
\be\label{B11}
\frac{1}{2}\|f\|_p^p \le \frac{1}{m}\sum_{j=1}^m |f(\xi^j)|^p \le \frac{3}{2} \|f\|_p^p,\   \ \forall f\in \bX_\bN.
\ee
\end{Proposition}	

 Clearly,  Proposition \ref{BP3} is a better version of Proposition \ref{BP2}. However, there is no guarantee that   the set $\xi(m)$   in Proposition \ref{BP3}   has   the tensor product structure. It would be interesting to prove    an analog of Proposition~\ref{BP3} where   	the set $\xi(m)$ of points is    given by  a tensor product of subsets of $\Og_i$, $i=1,\cdots,s$ and the power of the log factor in the estimate \eqref{B10}  is independent of $s$.  An affirmative answer to this question would   yield a significant reduction  of the log factor  in  the estimate \eqref{A2.4} of  Theorem \ref{AT2}, as can be seen from the following simple  lemma.
  
  \begin{Lemma} \label{lem3.2}
 Let  $X_N\subset \C(\Og)$ be an 
  	$N$-dimensional  subspace satisfying 
  	\be\label{A3.11}
  	\|f\|_\infty \le  B N^{1/p}\|f\|_{L_p(\Og, \mu)}, \quad \forall f\in X_N,
  	\ee
  for some $1\leq p<\infty$, and  constant $B\ge 1$.  	  	Assume that $1\in X_N$ and   there exists a  positive constant $\alpha$ such that  for  each integer $s\ge 2$,  
  there exists a  finite  subset  $\Ld=\Ld_1\times\cdots\times  \Ld_s\subset  \Og^s$ such that  
   $
 |\Ld|= |\Ld_1|\cdots |\Ld_s| \le C(p,s,\al) (B^pN)^s \left(\log_2(2B  N)\right)^\alpha$
   and 
   	\begin{align*}
   \frac{1}{2}\|F\|_{L_p(\Og^s, \mu^s)}^p \le \frac{1}{|\Ld|}\sum_{\og\in\Ld} |F(\og)|^p \le \frac{3}{2}\|F\|_{L_p(\Og^s, \mu^s)}^p,\  \ \forall F\in\bX_{\bN},
   \end{align*}
   where 
   $
   \bX_{\bN} := \sp\{ f_1(\bx^1)\times\cdots\times f_s(\bx^s)\,:\, f_i\in X_N, i=1,\dots,s\}.
   $
    Then for  any $\delta\in (0, 1)$, there exist 
  	$$
  	m \le C(p,\delta,\al) B^{p}N(\log_2 (2BN))^\delta 
  	$$
  	points $\xi^j\in \Omega$, $j=1,\dots,m$,   such that for any $f\in X_N$
  	we have
  	$$
  	\frac12\|f\|_p^p \le \frac{1}{m}\sum_{j=1}^m |f(\xi^j)|^p \le \frac 32\|f\|_p^p.
  	$$
  	\end{Lemma}
\begin{proof} Let $s\ge 2$ be an integer, and let $\Ld=\Ld_1\times\cdots\times  \Ld_s$ be a finite subset of $\Og^s$ with the stated properties in Lemma \ref{lem3.2}. 
Without loss of generality, we may assume that $|\Ld_1|=\min_{1\leq i\leq s} |\Ld_i|$. Then 
\[ |\Ld_1|\leq C(p, s,\al)^{1/s}   (B^p N) (\log_2 (2BN))^{\al/s}.\]
Let  $f\in X_N$, and define  
$$F(\bx^1,\cdots, \bx^s):=f(\bx^1),\   \   \    \bx:=(\bx^1,\cdots, \bx^s) \in\Og^s.$$
 Clearly, 
$\|F\|_{L_p(\Og^s)}^p=\|f\|_{L_p(\Og)}^p $
and 
$$ \frac{1}{|\Ld|}\sum_{\og\in\Ld} |F(\og)|^p=\frac 1 {|\Ld_1|} \sum_{\og_1\in\Ld_1} |f(\og_1)|^p.$$	
Since $1\in X_N$, we have $F\in\bX_{\bN}$. It then follows that 
\[ \frac 12 \|f\|_p^p \leq \frac 1 {|\Ld_1|} \sum_{\og_1\in\Ld_1} |f(\og_1)|^p \leq \frac 32 \|f\|_p^p.\]
	\end{proof}

\section{Sampling recovery}
\label{sr}

We first recall the setting 
 of the optimal recovery. For a fixed integer $m$ and a set of points  $\xi:=\{\xi^j\}_{j=1}^m\subset \Omega$, let $\Phi$ be a linear operator from $\bbC^m$ into $L_p(\Omega,\mu)$.
For a class $\bF\subset L_p(\Og, \mu)$ (usually, centrally symmetric and compact subset of $L_p(\Omega,\mu)$), define
$$
\varrho_m(\bF,L_p) := \inf_{\text{linear}\, \Phi; \,\xi} \sup_{f\in \bF} \|f-\Phi(f(\xi^1),\dots,f(\xi^m))\|_p.
$$
The above described recovery procedure is a linear procedure. 
The following modification of the above recovery procedure is also of interest. We now allow any mapping $\Phi : \bbC^m \to X_N \subset L_p(\Omega,\mu)$, where $X_N$ is a linear subspace of dimension $N\le m$,  and define
$$
\varrho_m^*(\bF,L_p) := \inf_{\Phi; \xi; X_N, N\le m} \sup_{f\in \bF}\|f-\Phi(f(\xi^1),\dots,f(\xi^m))\|_p.
$$

In both of the above cases we build an approximant, which comes from a linear subspace of dimension at most $m$. 
It is natural to compare the  quantities $\varrho_m(\bF,L_p)$ and $\varrho_m^*(\bF,L_p)$ with the 
Kolmogorov widths. Let $\bF\subset L_p$ be a centrally symmetric compact. The quantities  
$$
d_n (\bF, L_p) := \operatornamewithlimits{inf}_{\{u_i\}_{i=1}^n\subset L_p}
\sup_{f\in \bF}
\operatornamewithlimits{inf}_{c_i} \left \| f - \sum_{i=1}^{n}
c_i u_i \right\|_p, \quad n = 1, 2, \dots,
$$
are called the {\it Kolmogorov widths} of $\bF$ in $L_p$. In the definition of
the Kolmogorov widths we take the element of best
approximation of $f\in \bF$ as an approximating element
from $U := \sp \{u_i \}_{i=1}^n$.  This means
that in general (i.e. if $p\neq 2$) this method of approximation is not linear.

We have the following obvious inequalities
\be\label{I1}
d_m (\bF, L_p)\le \varrho_m^*(\bF,L_p)\le \varrho_m(\bF,L_p).
\ee

The main result of the paper \cite{VT183} is the following general inequality.
\begin{Theorem}\label{srT1}  There exist two positive absolute constants $c$ and $C$ such that for any   compact subset $\Omega$  of $\R^d$, any probability measure $\mu$ on it, and any compact subset $\bF$ of $\C(\Omega)$ we have
$$
\ro_{cn}(\bF,L_2(\Omega,\mu)) \le Cd_n(\bF,L_\infty).
$$
\end{Theorem}

We now formulate a conditional result from \cite{VT183}, which was used for the proof of Theorem \ref{srT1}. Let $X_N$ be an $N$-dimensional subspace of the space of continuous functions $\C(\Omega)$. For a fixed $m$ and a set of $m$  points  $\xi:=\{\xi^\nu\}_{\nu=1}^m\subset \Omega$,  we associate  a function $f\in \C(\Omega)$ with a vector
$$
S(f,\xi) := (f(\xi^1),\dots,f(\xi^m)) \in \bbC^m.
$$
Define
$$
\|S(f,\xi)\|_p:= \left(\frac{1}{m}\sum_{\nu=1}^m |f(\xi^\nu)|^p\right)^{1/p},\quad 1\le p<\infty,
$$
and 
$$
\|S(f,\xi)\|_\infty := \max_{\nu}|f(\xi^\nu)|.
$$
For a positive weight $\bw:=(w_1,\dots,w_m)\in \R_+^m$,  consider the following norm
$$
\|S(f,\xi)\|_{p,\bw}:= \left(\sum_{\nu=1}^m w_\nu |f(\xi^\nu)|^p\right)^{1/p},\quad 1\le p<\infty.
$$
Define the best approximation of $f\in L_p(\Omega,\mu)$, $1\le p\le \infty$ by elements of $X_N$ as follows
$$
d(f,X_N)_p := \inf_{u\in X_N} \|f-u\|_p.
$$
It is well known that there exists an element, which we denote by $P_{X_N,p}(f)\in X_N$, such that
$$
\|f-P_{X_N,p}(f)\|_p = d(f,X_N)_p.
$$
The operator $P_{X_N,p}: L_p(\Omega,\mu) \to X_N$ is called the Chebyshev projection. 

 Theorem \ref{srT2} below was proved in \cite{VT183} under the following assumptions.

{\bf A1. Discretization.} Let $1\le p\le \infty$. Suppose that $\xi:=\{\xi^j\}_{j=1}^m\subset \Omega$ is such that for any 
$u\in X_N$ in the case $1\le p<\infty$ we have
$$
C_1\|u\|_p \le \|S(u,\xi)\|_{p,\bw}  
$$
and in the case $p=\infty$ we have
$$
C_1\|u\|_\infty \le \|S(u,\xi)\|_{\infty}  
$$
with a positive constant $C_1$ which may depend on $d$ and $p$. 

{\bf A2. Weights.} Suppose that there is a positive constant $C_2=C_2(d,p)$ such that 
$\sum_{\nu=1}^m w_\nu \le C_2$.

Consider the following well known recovery operator (algorithm) 
$$
\ell p\bw(\xi)(f) := \ell p\bw(\xi,X_N)(f):=\text{arg}\min_{u\in X_N} \|S(f-u,\xi)\|_{p,\bw}.
$$
Note that the above algorithm $\ell p\bw(\xi)$ only uses the function values $f(\xi^\nu)$, $\nu=1,\dots,m$. In the case $p=2$ it is a linear algorithm -- orthogonal projection with respect 
to the norm $\|\cdot\|_{2,\bw}$. Therefore, in the case $p=2$ approximation error by the algorithm $\ell 2\bw(\xi)$ gives an upper bound for the recovery characteristic $\ro_m(\cdot, L_2)$. In the case $p\neq 2$ approximation error by the algorithm $\ell p\bw(\xi)$ gives an upper bound for the recovery characteristic $\ro_m^*(\cdot, L_p)$.

\begin{Theorem}\label{srT2} Under assumptions {\bf A1} and {\bf A2} for any $f\in \C(\Omega)$ we have
for $1\le p<\infty$
$$
\|f-\ell p\bw(\xi,X_N)(f)\|_p \le (2C_1^{-1}C_2^{1/p} +1)d(f, X_N)_\infty.
$$
Under assumption {\bf A1} for any $f\in \C(\Omega)$ we have
$$
\|f-\ell \infty(\xi,X_N)(f)\|_\infty \le (2C_1^{-1}  +1)d(f, X_N)_\infty.
$$
\end{Theorem}

Theorem \ref{srT2} is devoted to recovery by weighted least squares algorithms $\ell p \bw(\xi)$.
 It requires a discretization theorem
with positive  weights $\bw $. There is such a theorem   from \cite{LT}
for  $p=2$ and a general  subspace $X_N$ of $L_2(\Og, \mu)$, which we formulate as follows.

\begin{Theorem} \label{thm-4-3} There exist three absolute positive  constants $C_0, c_0, C_0'$ such that for every $N$-dimensional subspace $X_N$ of $L_2(\Og, \mu)$, there exist $m\leq C_0' N$ points $\xi^1,\cdots, \xi^m\in\Og$ and positive weights $w_1,\cdots, w_m$ such that 
	\begin{align}\label{weightedMZ}
	c_0	\|f\|_{2}^2 \leq   \sum_{j=1}^mw_j |f(\xi^j)|^2\leq  C_0 \|f\|_2^2,\   \  \forall f\in X_N.
	\end{align}	
\end{Theorem}

For a fixed integer $m\ge 1$ and a class $\bF\subset \C(\Omega)$ (usually, a centrally symmetric compact in $\C(\Omega)$), we define 
$$
\varrho_m^{wls}(\bF,L_2) := \inf_{\xi,\bw,  X_N, N\leq m}\sup_{f\in \bF} \|f-\ell 2 \bw(\xi,X_N)(f)\|_2,
$$
where the infimum  is  taken over  all  $N$-dimensional subspaces $X_N\subset L_2(\Og, \mu)$ with $N\leq m$, all collections $\xi:=\{\xi^j\}_{j=1}^m\subset \Omega$ of $m$ 
points in $\Og$,  and all positive weights $\bw=(w_1,\cdots, w_m) \in\mathbb{R}_+^m$.

Note that if we assume in addition that $1\in X_N$ in Theorem \ref{thm-4-3},  then the weights $w_j$ in \eqref{weightedMZ} satisfy 
$\sum_{j=1}^m w_j \leq C_0$. As a result, Theorem \ref{thm-4-3} combined with Theorem \ref{srT2} gives an analog of the following  Theorem \ref{IT2} from \cite{VT183}. 
\begin{Theorem}\label{IT2} There exist two positive absolute constants $c$ and $C$ such that for any   compact subset $\Omega$  of $\R^d$, any probability measure $\mu$ on it, and any compact subset $\bF$ of $\C(\Omega)$ we have
	$$
	\ro_{cn}^{wls}(\bF,L_2(\Omega,\mu)) \le Cd_n(\bF,L_\infty).
	$$
\end{Theorem}

 We may want  the recovery algorithm $\ell 2 \bw(\xi)$ to be the  classical least square 
algorithm, i.e. $\bw=\bw_m:=(1/m,\dots,1/m)$.
For that we need an $L_2$- discretization theorem with equal weights.
 There is such a theorem from \cite{LT}
under an extra assumption on the subspace $X_N$:

{\bf Condition E($t$).} We say that an orthonormal system $\{u_i(\bx)\}_{i=1}^N$ defined on $\Omega$ satisfies Condition E($t$) with a constant $t>0$ if for all $\bx\in \Omega$
$$
\sum_{i=1}^N |u_i(\bx)|^2 \le Nt^2.
$$

Under Condition E($t$), we have the following discretization  theorem, which was proved in  \cite{LT}:

\begin{Theorem}\label{CT2} Let  $\Omega\subset \R^d$ be a compact set with the probability measure $\mu$. Assume that $\{u_i(\bx)\}_{i=1}^N$ is a real (or complex) orthonormal system in $L_2(\Omega,\mu)$ satisfying Condition E($t$) for some constant $t>0$. 
	Then  there exists a set $\{\xi^j\}_{j=1}^m\subset \Omega$ of $m \le C_1 t^2 N$ points such that for any $f=\sum_{i=1}^N c_iu_i$  we have  
	\begin{equation*}
	C_2 \|f\|_2^2 \le \frac{1}{m}\sum_{j=1}^m |f(\xi^j)|^2 \le C_3 t^2\|f\|_2^2, 
	\end{equation*}
	where $C_1, C_2$ and $C_3$ are absolute positive constants. 
\end{Theorem}

Recall that  $\bw_m=(\frac1m,\cdots, \frac1m) \in\mathbb{R}^m$. 
 For a fixed positive integer $m$ and  a class $\bF\subset \C(\Og)$ (usually, a centrally symmetric compact in $\C(\Og)$), define
$$
\varrho_m^{ls}(\bF,L_2) := \inf_{\xi,\,X_N, N\leq m} \sup_{f\in \bF} \|f-\ell 2 \bw_m(\xi,X_N)(f)\|_2,
$$
where     the infimum is   taken over all  $N$-dimensional subspaces $X_N\subset L_2(\Og, \mu)$ with $N\leq m$, and  all collections $\xi:=\{\xi^j\}_{j=1}^m\subset \Omega$ of $m$ 
points in $\Og$.
We now define $E(t)$-conditioned Kolmogorov widths by 
$$
d_N^{E(t)}(\bF,L_p) :=     \inf_{\{u_1,\dots,u_N\}\, \text{satisfies Condition} E(t)}   \sup_{f\in \bF}\inf_{c_1,\dots,c_N}\|f-\sum_{i=1}^N c_iu_i\|_p.
$$

Now combining Theorem \ref{CT2}  with Theorem \ref{srT2}, we obtain the following result, which was proved in  \cite{VT183}. 
\begin{Theorem}\label{CT3} Let $\bF$ be a compact subset of $\C(\Omega)$. There exist two positive   constants $c$ and $C$ which may depend on $t$ such that
$$
\ro_{cn}^{ls}(\bF,L_2) \le Cd_n^{E(t)}(\bF,L_\infty).
$$
\end{Theorem}

We now present some results on the sampling recovery in $L_p$, $2<p<\infty$. 
For a fixed positive integer  $m$,  and  a class $\bF\subset \C(\Og)$ (usually, a centrally symmetric compact in $\C(\Og)$), we define 
$$
\varrho_m^{lp}(\bF,L_p) := \inf_{\xi,\,X_N, N\leq m} \sup_{f\in \bF} \|f-\ell p \bw_m(\xi,X_N)(f)\|_p,
$$
 where   the infimum is   taken over all  $N$-dimensional subspaces $X_N\subset L_p(\Og, \mu)$ with $N\leq m$, and  all collections $\xi:=\{\xi^j\}_{j=1}^m\subset \Omega$ of $m$ 
points in $\Og$.

{\bf Condition NpB.} We say that an $N$-dimensional subspace $X_N\subset L_p(\Omega,\mu)$ satisfies Condition NpB (Nikol'skii-type inequality for the pair $(p,\infty)$) if for all $\bx\in \Omega$ we have for all $f\in X_N$
$$
|f(\bx)| \le BN^{1/p}\|f\|_{p},
$$
where $B\ge 1$ is a constant. 
It is well known that Condition E($t$) is equivalent to Condition N2$t$ (see, for instance, \cite{LT} for an explanation and \cite{DP} for a detailed discussion).

We now define the NpB-conditioned Kolmogorov width by 
$$
d_N^{NpB}(\bF,L_p) :=     \inf_{X_N\, \text{satisfies Condition NpB} }  \sup_{f\in \bF}\inf_{g\in X_N}\|f-g\|_p.
$$

Theorem \ref{AT2} combined with Theorem \ref{srT2} gives the following analog of Theorem \ref{CT3}.

\begin{Theorem}\label{srT4} Let $\bF$ be a compact subset of $\C(\Omega)$.	
	For $p\in (2,\infty)$ and every constant $B\ge 1$,  there exist two positive   constants $c$ and $C$, which may depend on $p$, such that
$$
\ro_{m}^{lp}(\bF,L_p) \le Cd_n^{NpB}(\bF,L_\infty)
$$
provided $m\ge cB^pn(\log (2B n))^3$.
\end{Theorem}

Corollary \ref{AC2} combined with Theorem \ref{srT2} gives the following generalization of Theorem \ref{srT4}.

\begin{Theorem}\label{srT4g} Let $\bF$ be a compact subset of $\C(\Omega)$, and let  $2\le q\le p< \infty$. Let $B\ge 1$ be a given constant.  Then  there exist two positive   constants $c$, $C$, which may depend on $q,p$, such that
$$
\ro_{m}^{lp}(\bF,L_p) \le Cd_n^{NqB}(\bF,L_\infty)
$$
provided $m\ge cB^pn^{p/q}(\log (2B n))^3$.
\end{Theorem}

{\bf Condition EqB.} We say that an $N$-dimensional subspace $X_N\subset L_p(\Omega,\mu)$ satisfies Condition EqB (Entropy condition with parameters $q$ and $B$) if it satisfies inequalities (\ref{A1}).

Using Corollary \ref{AC1} and Theorem \ref{srT2} we obtain the following version of Theorem \ref{srT4g}.

\begin{Theorem}\label{srT5} Let $\bF$ be a compact subset of $\C(\Omega)$ and let $1\le q\le p<\infty$. There exist two positive   constants $c$ and $C$ which may depend on $p$ and $q$ such that
$$
\ro_{m}^{lp}(\bF,L_p) \le Cd_n^{EqB}(\bF,L_\infty)
$$
provided $m\ge cB^pn^{p/q}(\log(2Bn))^2$.
\end{Theorem}

\Addresses

\end{document}